\documentclass[12pt]{amsart}
\usepackage{amsmath,latexsym,amssymb}

\newcommand{\rar}{\rightarrow}
\newcommand{\lar}{\longrightarrow}

\newtheorem{Theorem}{Theorem}[section]

\newtheorem{Corollary}[Theorem]{Corollary}

\newtheorem{Proposition}[Theorem]{Proposition}
\newtheorem{proposition}[Theorem]{Proposition}
\newtheorem{Conjecture}[Theorem]{Conjecture}

\newtheorem{Remark}[Theorem]{Remark}
\newtheorem{remark}[Theorem]{Remark}

\newtheorem{example}[Theorem]{Example}
\newtheorem{Definition}[Theorem]{Definition}

\theoremstyle{plain}

\newtheorem*{thm*}{Theorem}
\newtheorem*{cor*}{Corollary}

\newtheorem*{claim*}{Claim}

\theoremstyle{definition}

\theoremstyle{remark}

\numberwithin{equation}{Theorem}

\def\ann{\mbox{\rm ann}}

\def\hdeg{\mbox{\rm hdeg}}
\def\Hom{\mbox{\rm Hom}}
\def\End{\mbox{\rm End}}
\def\coker{\mbox{\rm coker }}
\def\Ext{\mbox{\rm Ext}}
\def\Tor{\mbox{\rm Tor}}
\def\depth{\mbox{\rm depth }}

\def\height{\mbox{\rm height }}

\def\bar#1{{\overline{#1}}}

\def\xx{{\bf x}}
\def\hh{{\bf h}}
\def\ff{{\bf f}}

\def\RR{{\bf R}}

\def\CC{{\Lambda}}
\def\II{{\bf I}}
\def\JJ{{\bf J}}

\def\SS{{\bf S}}
\def\ZZ{{\bf Z}}
\def\FF{{\bf F}}

\def\m{{\mathfrak m}}
\def\p{{\mathfrak p}}
\def\LL{{\mathbf L}}
\def\KK{{\mathbf K}}

\begin{document}

\title[On the   Radical    of Endomorphism  Rings  of Local Modules
]{
On  the Radical    of Endomorphism Rings  of Local Modules}

\author{Wolmer V. Vasconcelos}
\address{Department of Mathematics, Rutgers University, 110 Frelinghuysen Rd, Piscataway, NJ 08854-8019, U. S. A.}
\email{vasconce@math.rutgers.edu}

\thanks{{AMS 2010 {\em Mathematics Subject Classification:} Primary:
13H10, Secondary: 13C14, 16E65.}
 }

\begin{abstract}
\noindent
We study the construction and homological properties of modules whose rings of
endomorphisms have a unique two-sided maximal ideal.
\end{abstract}

\maketitle

\begin{center}
Dedicated to Aron Simis on the occasion of his $70$th birthday.
\end{center}

\section{Introduction}

Let
 $(\RR, \m)$   be a Noetherian local ring and $E$ a finitely
generated $\RR$-module. The ring $\CC=\Hom_{\RR}(E,E)=\End(E,E)$ expresses many
properties of $E$, and sometimes of $\RR$ itself,  many of which are well hidden.
The terminology {\em local module } refers to a finitely generated
module $E$ over a  Noetherian local ring $(\RR, \m)$ such that
the ring of endomorphisms $\End_{\RR}(E)=\Hom_{\RR}(E,E)$ has a unique two-sided
maximal ideal. 
 Our aim in this note is to develop techniques of construction of 
 local modules
 and study their homological properties.  It
is also motivated by their roles in the following related 
 questions:

\begin{itemize}
\item The HomAB question:
Can the number of generators $\nu(\CC)$ of $\CC$ be estimated in terms of
properties of $E$? For a detailed discussion of this question we
refer
 to \cite{Dalili}, \cite{DV}. Recently, Kia Dalili resolved one its
main problems in proving the existence, in the graded case,  of {\it a priori} uniform 
bounds for $\nu(\CC)$ in terms of extended cohomological degrees of
$E$.

\item Non-commutative desingularization:
There are
 noteworthy rings of endomorphisms
 in the recent  literature. One class of the most
intriguing rings are those of finite global dimension; see  \cite{DH10},
\cite{Le11},
\cite{denBergh1}, \cite{denBergh2}.
\index{noncommutative
regular ring} We show this is rare for local modules.

\item Degree representation of a module question: For a module $E$, its degree
representation is the smallest integer $r$--when it exists--such that
there is an embedding of $\RR$-algebras $\varphi: \Hom_{\RR}(E,E)\lar
 M_r(\RR)$. Which modules admit such a degree?

\end{itemize}

 We want to study $\CC$ by 
 examining its  Jacobson radical. For simplicity we often denote  the module of $\RR$-homomorphisms
$\Hom_{\RR}(E,F)$ by $\Hom(E,F)$ and employ the similar notation for
some derived functors. We will argue that
local modules often work as
  blocks  in building other modules
whose modules of endomorphisms allow for the determination of their
Jacobson radicals.

In sections 2 and 3 we develop techniques to  tell when modules
of syzygies are local. The most interesting classes of such modules
avail themselves of the notion of the Auslander dual (\cite{AusBr}).
 The main results of this note are given in
the last section, particularly in Theorem~\ref{homollocal}, showing
that if $\RR$ is normal, contains the rationals and $E$ is torsionfree,
then if $E$ is a local module that is $\CC$-projective then $E$ is
$\RR$-free.

\section{Jacobson radical}

In this section and next,
we  treat conditions for the algebra $\CC=\Hom_{\RR}(E,E)$
to have a unique two-sided maximal ideal. An example is a free
$\RR$-module.
An extreme instance are modules whose rings of endomorphisms are local, that is have a unique maximal 
ideal.  There is an obvious 
 separation between the two cases. 
Local modules come however in greater variety. For instance, if $K$ is a field and
$\RR= K[\epsilon]$, $\epsilon^2=0$, is the ring of {\em dual numbers}, then any finitely generated $\RR$-module
$E = \RR^m \oplus (\epsilon)\RR^{ n}$. 
Observing that
by Nakayama Lemma, $\m \CC$ is contained in the Jacobson radical
$\JJ$ of $\CC$, all these modules are local.

\medskip

We will denote the $\RR$-dual of an $\RR$-module $E$ by $E^*=\Hom_{\RR}(E,\RR)$.
 These modules admit a natural $\CC$ action,  
 $E$ as a left $\CC$-module and
$E^*$ as a right $\CC$-module.
 We
say that $E$ is reflexive if the usual mapping $E\rar E^{**}$ is an
isomorphism. A key role will be played in our discussion by the module $E^*\otimes_{\RR} E$, which carries two
natural structures, as a left and a right $\CC$-module.

\begin{proposition}\label{jrad1} Let $(\RR, \mathfrak{m})$ be a Noetherian local
ring and let $E$ be a finitely generated $\RR$-module.
\begin{enumerate}
\item[{\rm (i)}]
 If $E$ has no free summand, then the image of $E^*\otimes E$ in
$\CC$ is a two-sided ideal contained in the Jacobson radical $\JJ$ of $\CC$.
\item[{\rm (ii)}]  $\Hom_{\RR}(E, \m E)$ is a two-sided ideal contained in $\JJ$.
\end{enumerate}
\end{proposition}

\begin{proof} There are two actions of $\CC$ on $E^*\otimes E$. For
$\hh\in \CC$, $(f\otimes e)\hh= f\circ \hh \otimes e$ and
$\hh(f\otimes e)= f\otimes \hh(e)$.

 Let $\II$ be the identity of $\CC$. To prove that
\[ \hh=\II + \sum_{i=1}^n f_i\otimes e_i\]
is invertible, note that for each $e\in E$,
\[ \hh(e)=e + \sum_{i=1}^n f_i(e) e_i\in e + \mathfrak{m}E,\]
since $f_i(e) \in \mathfrak{m}$ as $E$ has no free summand.
From the Nakayama Lemma, it follows that $\hh$ is a surjective
endomorphism, and therefore must be invertible.

\medskip

The proof of (ii) is similar.
\end{proof}

Throughout $(\RR,\m)$ is a Noetherian local ring and $E$ is a
finitely generated $\RR$-module.
The ideal $E^*(E)= \tau(E)=(f(e), f\in E^*, e\in E)$ is the {\em trace
ideal}\index{trace ideal} of $E$. It is of interest to describe classes of modules of
syzygies with $\tau(E)\subset \m$.
If the trace ideal
$E^*(E)=\RR$, by Morita's theory (see \cite[Theorem A.2]{AG}), $E$ is a
projective module over $\CC=\End(E)$. This is easy to verify directly
as follows. Suppose $E=\RR\epsilon\oplus M$. Then $\CC=\End(E)$ may
be represented as
\[
 \CC=
\left[\begin{array}{cc}
R & M^* \\
M & \End(M)
\end{array}\right].
\] As a left $\CC$-module, $E=\CC\epsilon$ where the annihilator of
$\epsilon$ is the left ideal
\[
 L=
\left[\begin{array}{cc}
0 & M^* \\
0 & \End(M)
\end{array}\right],
\]
which splits off $\CC$.

\begin{remark}{\rm
There are cases however when
$E$ is $\CC$-projective but $E^*(E)\neq \RR$. For instance, for any local Noetherian
domain of Krull dimension $1$ and multiplicity $\leq 2$,
 then any ideal has this property (\cite{Bass}).
}
\end{remark}

\begin{remark}
From the sequence
\[ 0 \rar \m E \lar E \lar E/\m E=\bar{E}= k^n \rar 0,
\]
we have the exact sequence
\[
0 \rar \Hom_{\RR}(E, \m E) \lar \CC \lar M_n(k) \lar \Ext_{\RR}^1(E, \m E).
\]
This shows that $\dim \CC/\JJ\leq n^2$, since $\Hom_{\RR}(E,\m
E)\subset \JJ$, giving an explicit  control over the number of
maximal two-sided ideals of $\CC$.
\end{remark}

If $E$ is a module without free summand, we use the notation of the
subideals of $\JJ$
\[\JJ_1=\mbox{\rm image  $E^*\otimes E $}\subset
 \JJ_0=\Hom(E, \m E) \subset \JJ.\]

We introduce some terminology on special homomorphisms.
Let $(\RR,
\m)$ be Noetherian local ring, and let  $E$ and $F$ be finitely generated
$\RR$-modules. A homomorphism $\varphi: E\rar F$ is said to be {\em
small}\index{small homomorphism} if $\varphi(E)\subset \m F $.

\medskip

In assembling the Jacobson radical of the ring of endomorphisms of
direct sums we will consider conditions such as:
\begin{itemize}
\item All homomorphisms in $\Hom(E,F)$ are small.

\item The endomorphisms of $E$ that factor through $F$ are small,
that is
\[\Hom(F, E)\circ \Hom(E,F)\subset \Hom(E, \m E).\]

\end{itemize}

\subsection*{Decomposition}

Let $(\RR,\m)$ be a local ring and $E$ a finitely generated $\RR$-module.
Suppose $E=E_1\oplus E_2$ is a nontrivial decomposition. We seek
 to relate the Jacobson of $\CC=\Hom(E,E)$ to those of the
subrings $\CC_i=\Hom(E_i,E_i)$, $i=1,2$ and the relationships
between $E_1$ and $E_2$.
A special   case is that
where $E_1= F=\RR^n$ and $E_2=M$  is a
module without free summands.

\begin{Proposition} Let $\RR$ be a local ring and $E$ a  module with
a decomposition $E=E_1\oplus E_2$. Set $\CC_i=\Hom_\RR(E_i, E_i)$ and
$\JJ'_i=\Hom_\RR(E_i, \m E_i)$. Assume the transition conditions
\[
\Hom_\RR(E_j, E_i)\cdot \Hom_\RR(E_i,E_j)\subset
\Hom_\RR(E_j, \m E_j), \quad i\neq j.\] Then the Jacobson
radical of $\CC=\End(E)$ is
\[ \JJ=\left[ \begin{array}{cc}
\JJ_1 &\Hom(E_2,E_1) \\
\Hom(E_1, E_2) & \JJ_2
\end{array}
\right],\]
where  $\JJ_i$ is the Jacobson radical of
$\End(E_i)$.
\end{Proposition}

\begin{proof} Note that $\JJ_i'$ is a subideal of $\JJ_i$, and for
$A=\Hom(E_1, E_2)$ and $B=\Hom(E_2, E_1)$ it holds that $A\cdot
B\subset \JJ_2'$ and $B\cdot A\subset \JJ_1'$. From these it follows
that
\[ \LL=\left[ \begin{array}{cc}
 \JJ_1 &  B \\
A &  \JJ_2
\end{array}
\right]\] is a two-sided ideal of $\CC$.

Denote by $\II$ and $\II_i$  the identities  of $\CC$ and $\CC_i$.
To show $\LL$ is the Jacobson radical it is enough to show that for $\Phi\in
\LL$, $\II+\Phi$ is invertible, or equivalently it is a surjective
endomorphism of $E$. In other words, for $a\in \JJ_1$, $b\in A$,
$c\in B$ and $d\in \JJ_2$ the system of equations
\begin{eqnarray*}
x + ax + by &=& u\\
c x + y + dy &=& v
\end{eqnarray*}
for $u\in E_1$, $v\in E_2$ has always a solution $x\in E_1$ and $y\in
E_2$.
It is enough to observe that in the formal solution
\begin{eqnarray*}
x & = & (\II_1 + a)^{-1}(u-by)\\
y &=& (\II_2 + d -c(\II_1 + a)^{-1}b)^{-1}(v-c(\II_1+a)^{-1}u),
\end{eqnarray*}
$c(\II_1+a)^{-1}b\in \JJ_2$.

\end{proof}

\begin{Corollary}\label{Jacsemi} Let $\RR$ be a local ring and $M$ a  module without
free summands.  If $F$ is a free $\RR$-module of positive rank then the Jacobson
radical of $\CC=\End(F\oplus M)$ is
\[ \JJ=\left[ \begin{array}{cc}
\m \cdot \End(F) & \Hom(M,F) \\
\Hom(F,M) & \JJ_0
\end{array}
\right],\]
where  $\JJ_0$ is the Jacobson radical of
$\End(M)$.
\end{Corollary}

\begin{proof} Note that $\Hom(M,F) = \Hom(M, \m F)$.
\end{proof}

This shows that a non-free  module with a free summand is never local.
It permits, to understand the Jacobson radical of $\End(E)$, to peel
away from $E$ a free summand of maximal rank.

\subsection*{Modules without free summands}

Let $(\RR,\m)$ be a Cohen--Macaulay local ring of dimension $d$ and
$E$ a finitely generated $\RR$--module.

\begin{example}{\rm
We will now describe two classes of modules of syzygies that do not
have free summands.

\medskip

\begin{enumerate}

\item {\cite[Lemma 1.4]{Herzog78}} Let $\RR$ be a Cohen--Macaulay
local ring and $E$ a module of
syzygies
\[ 0 \rar E \lar F \lar M \rar 0,\]
$E\subset \m F$ and $M$ a maximal Cohen--Macaulay module. Then $E$
has no free summand.

\begin{proof} Assume otherwise that $E=\RR\epsilon \oplus E'$, $0\neq \epsilon\in \m F$.
Setting $M'=F/\RR\epsilon$ we have the exact sequence
\[ 0 \rar E' \lar M' \lar M \rar 0 \] showing that $M'$
is a maximal
Cohen--Macaulay module. But $M'$ is a module with a free resolution, so must be free
by Auslander-Buchsbaum equality (\cite[Theorem 1.3.3]{BH}).
 This means that $\RR\epsilon$ is a summand of $F$, which is impossible.

\end{proof}

\item Suppose $\RR$ is a Gorenstein local ring and $E$ is a module of
syzygies
\[ 0 \rar E \stackrel{\varphi}{\lar} F_s \lar \cdots \lar F_1 \lar F_0 \lar M \rar 0,
\] $E\subset \m F_s$. If $1\leq s\leq \height(\ann(M))-2$ then $E$
has no free summand.

\begin{proof}  $E$ is a reflexive being a second syzygy module over
the Gorenstein ring $\RR$. It will be enough to show the dual module
of $\RR\epsilon$ splits off $F_s^*$.

\medskip

By assumption $\Ext_{\RR}^i(M,\RR)=0$ for $i\leq s+1$. Applying
$\Hom(\cdot, \RR)$ to the complex we obtain an exact complex
\[ 0 \rar F_0^* \lar F_1^* \lar \cdots \lar F_s^*
\stackrel{\varphi^*}{\lar E^*} \rar 0,
\]
to prove the assertion since $\varphi=\varphi^{**}$.
\end{proof}

The assumption that $\RR$ is Gorenstein was used to guarantee that
second syzygies modules are reflexive. This could be achieved in many
other ways. For example, with $\RR$ Cohen--Macaulay and $\height
\ann(M)\geq 2$.

\medskip

\item Taken together these two observations give an overall picture
of the syzygies of $M$ which are without free summands.
 A special case is that of the syzygies in the Koszul complex of
a regular sequence.

\end{enumerate}
}
\end{example}

\section{Construction of local modules} \index{construction of
local modules}

For an $\RR$-module $A$, its ring of endomorphisms $\End_\RR(A)$
arises out of the syzygies of $A$ in the following well-known manner.


\begin{Proposition} \label{ausdual2} Let $A$ be a finitely generated
$\RR$-module with a presentation
\[ \RR^m \stackrel{\varphi}{\lar} \RR^n \lar A \rar 0.\]
Then $\Hom_{\RR}(A,A)$ is isomorphic to the kernel of the
induced mapping
\[\Phi_A: A\otimes_{\RR} (\RR^n)^{*} \stackrel{\II\otimes \varphi^{*}}{\lar}
 A\otimes (\RR^m)^{*}.\]
\end{Proposition}
%
%


\subsection*{Auslander dual}
Our main vehicle to build local modules out of other local modules
 is the following notion of
Maurice Auslander (\cite{AusBr}).

\begin{Definition} \label{ausdualdef} {\rm Let $E$ be a finitely generated $\RR$-module with a
projective presentation
\[ F_1 \stackrel{\varphi}{\lar} F_0 \lar E \rar 0.\]
The {\em Auslander dual} of $E$ is the module $D(E)=
\coker(\varphi^{*})$,
\begin{eqnarray} \label{ausdual}
0\rar E^*\lar  F_0^* \stackrel{\varphi^{*}}{\lar} F_1^* \lar D(E) \rar 0.
\end{eqnarray}
}\end{Definition}

The module
$D(E)$ depends on the chosen presentation but it is unique up to
projective summands. In particular the values of the functors
$\Ext_{\RR}^i(D(E),\cdot)$ and $\Tor_i^{\RR}(D(E), \cdot)$, for $i\geq 1$,
are independent of the presentation. Its use here lies in the
following result (see \cite[Chapter 2]{AusBr}):

\begin{Proposition}\label{Adual}
 Let $\RR$ be a Noetherian ring and
 $E$  a finitely generated $\RR$-module. There are two exact
sequences of functors:
\begin{eqnarray} \label{adual1}
\hspace{.4in} 0 \rar \Ext_{\RR}^1(D(E),\cdot) \lar E\otimes_{\RR}\cdot \lar
\Hom_{\RR}(E^*,\cdot) \lar \Ext_{\RR}^2(D(E),\cdot)\rar 0
\end{eqnarray}
\begin{eqnarray}  \label{adual2}
\hspace{.4in}   0 \rar \Tor_2^{\RR}(D(E),\cdot) \lar E^*\otimes_{\RR}\cdot \lar
\Hom_{\RR}(E,\cdot) \lar \Tor_1^{\RR}(D(E),\cdot)\rar 0.
\end{eqnarray}
\end{Proposition}

The setup we employ is derived from the analysis of duality of
\cite{AusBr}: There is an exact complex of $\RR$-algebras
\begin{eqnarray} \label{ausbr0}
 E^* \otimes_{\RR} E \lar \Hom_{\RR}(E,E)=\End_{\RR}(E) \lar \Tor_1^{\RR}(D(E),E) \rar 0
\end{eqnarray}
where $D(E)$ is the Auslander dual of $E$. We note that $E^*\otimes_{\RR} E$ maps onto the endomorphisms of $E$
that factor through projective modules.
 
We now turn to 
the understanding of $\Tor_1^{\RR}(D(E),E)$ as a more amenable  ring of  endomorphisms. 
 According to
\cite[Theorem 1.3]{HiltonRees} there is a natural surjective 
homomorphism \[ \Hom_{\RR}(E,E) \mapsto  \Hom_{\RR}(\Ext_{\RR}^1(E,\cdot),\Ext_{\RR}^1(E,\cdot)).\]
Its kernel consists of the maps $f: E \rar E$ that factor through projective modules, a fact  that can be seen
with a brief calculation by considering a presentation $0 \rar M \rar P \rar E \rar 0$, $P$ projective,  and assuming 
for $f: E \rar E$ that $f^*: \Ext_{\RR}^1(E,M) \rar \Ext_{\RR}^1(E,M)$ vanishes. 
The useful point is that 
the  
 endomorphism ring 
 $\Hom_{\RR}(\Ext_{\RR}^1(E,\cdot),\Ext_{\RR}^1(E,\cdot))$ may have a much smaller support than $E$ itself. Furthermore
 if $E$ has no free summands, $E^*\otimes_{\RR} E$ maps into the Jacobson radical of $\Hom_{\RR}(E,E)$ and therefore
$E$ is a local module if 
$\Hom_{\RR}(\Ext_{\RR}^1(E,\cdot),\Ext_{\RR}^1(E,\cdot))$ has a unique maximal two-sided ideal.

\medskip

In each application of this setup we look for which of
 $\Hom_{\RR}(\Ext_{\RR}^1(E,\cdot),\Ext_{\RR}^1(E,\cdot))$ or
$\Tor_1^\RR(D(E),E)$ is more amenable. Let us consider several examples.

\subsection*{Syzygies of perfect modules} These modules, or mild generalizations of them, provide numerous
examples of local modules.

\begin{example} 
{\rm Let $\RR$ be a Cohen--Macaulay local ring of dimension
$d\geq 3$,
and  let $E$
 be a module defined by one relation
\begin{eqnarray}\label{onerelmod}
 0\rar \RR \stackrel{\varphi}{\lar} \RR^n \lar E \rar 0. 
\end{eqnarray}
We assume that the entries of $\varphi$ define an ideal $I$ of height
$\geq 3$, minimally generated by $n$ elements.  This makes $E$ a module that is free in codimension $\leq 2$
and satisfies the condition $S_2$ of Serre. With  these conditions $E$ is then a reflexive module 
that has no free summands.
Let us determine the
number of generators of $\CC=\Hom_{\RR}(E,E)$, and some of its other
properties.
\medskip

In this example, since $E$ has projective dimension $1$,
the functor $\Ext_{\RR}^1(E,X)= \Ext_{\RR}^1(E, \RR)\otimes X= \RR/I\otimes
X$, so that
  \[\Hom_{\RR}(\Ext_{\RR}^1(E,\cdot),\Ext_{\RR}^1(E,\cdot))=\Hom_{\RR}(\RR/I,\RR/I)=\RR/I.\]
Tensoring (\ref{onerelmod}) by $E^*$, we get the complex
\[ 0 \rar E^* \lar E^*\otimes \RR^n \lar E^*\otimes E \rar 0,\]
which is exact by acyclicity lemma.  Since $E^*$ is reflexive and $\RR$ is Cohen-Macaulay, $E^*$ has the
condition $S_2$ of Serre and therefore $E^*\otimes E$ has the condition $S_1$ of Serre and it is free in
codimension $\leq 1$, 
 in particular it
is torsionfree.
Since $E$ has no
free summand,  by Proposition~\ref{jrad1}  $E$ is a
local module. As for the number of generators of $\CC$,
\[\beta_0(I)\beta_1(I)-\beta_0(I)+1 \leq
 \nu(\CC)\leq \nu(E)\nu(E^*)+1= \beta_0(I)\beta_1(I)+1,\]
 where $\beta_i(\cdot)$ denotes Betti numbers.}
\end{example}

\begin{example}
{\rm Let $E$ be a module with a free resolution
\[0 \rar F_n \stackrel{\psi}{\lar} F_{n-1} \lar \cdots \lar F_1
\stackrel{\varphi}{\lar} F_0 \lar E \rar 0.\]
Assume that $\Ext_\RR^i(E,\RR)=0$, $1\leq i\leq n-1$. Dualizing we
have the exact complex
\[0 \rar E^*\lar
 F_0^* \stackrel{\varphi^*}{\lar} F_{1} \lar \cdots \lar F_{n-1}^*
\stackrel{\psi^*}{\lar} F_n^* \lar \Ext_\RR^n(E,\RR) \rar 0.\]

This gives that $D(E)=\coker \varphi^*$ and thus $\Tor^\RR_1(D(E),
E)=\Tor_n^\RR(\Ext_\RR^n(E,\RR),E)$. Note also that
$D(\Ext_\RR^n(E,\RR))= \coker \psi$. Now we apply
Proposition~\ref{ausdual2} to the module $A=\Ext_\RR^n(E,\RR)$ to get
isomorphisms
\[\Hom(\Ext_\RR^n(E,\RR),\Ext_\RR^n(E,\RR))=\ker \Phi_A =
\Tor_n^\RR(\Ext_\RR^n(E,\RR),E).
\]
This gives the exact sequence
\[ 0\rar E^*\otimes E \lar \End(E) \lar \End(\Ext_\RR(E,\RR))  \rar 0,
\] since $\Tor_2^\RR(D(E), E)=\Tor_{n+1}(\Ext_\RR^n(E,\RR), E)=0$.
}
\end{example}


Let us discuss an example that is very rich of details,
 the modules of cycles of a Koszul complex
$\KK(\xx)$ associated to a regular sequence $\xx=\{x_1, \ldots,
x_n\}$, $n\geq 5$:
\[ \KK(\xx): \quad 0 \rar K_n \rar K_{n-1} \rar K_{n-2} \rar \cdots
\rar K_2\rar K_1\rar K_0 \rar 0.\]
For simplicity we take for module $E$ the $1$-cycles $Z_1$ of $\KK$. Note that $Z_1$ is the second module
of syzygies of $\RR/(\xx)$.
There is a pairing in the subalgebra  $\ZZ$ of cycles leading 
\[ Z_1 \times Z_{n-2}\rar Z_{n-1} = \RR,\]
that identifies $Z_{n-2}$ with the dual $E^*$ of $E$.

We are now ready to put this data into the framework of the
Auslander dual. Dualizing the projective presentation of $E$,
\[ 0 \rar K_n \lar \cdots \lar K_{3}\lar K_{2} \lar E=Z_1 \rar 0,\]
gives us the exact  complex
\[0 \rar E^* \lar K_{2}^* \lar K_{3}^* \lar D(E) \rar 0.\]
In other words,  to the identification  $D(E)=Z_{n-4}$:
\[ 0 \rar Z_{n-2} \lar K_{n-2} \lar K_{n-3} \lar Z_{n-4}  \rar 0.\]

Now for the computation of $\Tor_1^{\RR}(D(E),E)$:

\begin{eqnarray*} \Tor_1^{\RR}(D(E),E)&=& \Tor_1^{\RR}(Z_{n-4},E) =
\Tor_2^{\RR}(Z_{n-5}, E) =
\cdots \\
&=& \Tor_{n-3}^{\RR}(Z_0,E)= \Tor_{n-2}^{\RR}(\RR/I,E) = K_{n}\otimes
\RR/I=\RR/I.
\end{eqnarray*}

\begin{Remark}{\rm The number of generators of $\Hom_{\RR}(E,E)$ is bounded
by
\[ \nu(E)\nu(E^*) + 1 .\]
}\end{Remark}

For the purpose of a comparison, let us evaluate $\hdeg (E)$, a quantity that 
controls the number of generators of $E^*$
 (for
which 
refer to \cite{Dalili}). The
multiplicity of $E$ is $(n-1)\deg(\RR)$. Applying $\Hom_{\RR}(\cdot, R)$ to the
projective resolution of $E$, we get
 $\Ext_{\RR}^{n-2}(E,\RR) =\Ext_{\RR}^n(\RR/I,\RR) = \RR/I $ is Cohen-Macaulay,
and its contribution in the formula for $\hdeg (E)$ becomes
\[ \hdeg (E) = \deg (E) + {{d-1}\choose{n-3}}
\hdeg (\Ext_{\RR}^{n-2}(E,\RR))= (n-1) +
{{d-1}\choose{n-3}}
\deg (\RR/I).
\]
It is clear that in appealing to \cite[Theorem 5.2]{Dalili}, to get
information about $\nu(E^{*})$, a similar calculation can be carried
out for any module of cycles of a projective resolution of broad
classes
of Cohen-Macaulay modules.

\medskip

Let $(\RR, \mathfrak{m})$ be a Gorenstein local ring of dimension $d$.
Let $M$ be a perfect $\RR$-module with a minimal free resolution

\[ \FF: \quad 0 \rar F_n \rar F_{n-1} \rar F_{n-2} \rar \cdots
\rar F_2\rar F_1\rar F_0 \rar M \rar 0.\]
We observe that dualizing $\FF$ gives a minimal projective resolution
$\LL$ of $\Ext_{\RR}^{n}(M,\RR)$.
Let $E$ be the module $Z_{k-1}=Z_{k-1}(\FF)$ of $k-1$-cycles of $\FF$,
\[ F_{k+1} \lar F_{k} \lar E \rar 0.\]
Dualizing, to define the Auslander dual $D(E)$, gives the complex
\[ 0 \rar E^{*} \lar L_{n-k} \lar L_{n-k-1} \lar D(E)\rar 0.\]
It identifies $E^*$ with the $(n-k)$--cycles of $\LL$, and $D(E)$ with
its $(n-k-2)$--cycles. In particular this gives $\nu(E)=
\beta_{k}(M)$ and $\nu(E^*)=\beta_{k-1}(M)$.

\medskip

Now
to make use of the Auslander dual setup, we seek some
control over $\Tor_1^{\RR}(D(E),E)$. We make use first of the complex
\[ 0 \rar D(E) \lar L_{n-k-2} \lar \cdots \lar L_0 \lar
\Ext_{\RR}^n(M,\RR)
\rar 0,\]
to get
\[ \Tor_1^{\RR}(D(E),E) \simeq \Tor_{n-k}^{\RR}(\Ext^n(M,\RR),E), \]
and then of the minimal resolution of $E$  to obtain
\[ \Tor_1^{\RR}(D(E),E) \simeq \Tor_{n}^{\RR}(\Ext^n(M,R),M). \]
In particular, $\Tor_1^{\RR}(D(E),E)$ is independent of which module of
syzygies was taken. Furthermore, the calculation  shows that
 $\Tor_2^{\RR}(D(E),E)=0$.

Placing these elements together, we have the exact sequence
\[ 0 \rar E^*\otimes E \lar \Hom_{\RR}(E,E) \lar
\Tor_n^{\RR}(\Ext_{\RR}^n(M,\RR),M)
\rar 0.\] Note that $E^*\otimes E$ is a torsion free $\RR$-module.

Finally, it follows from Proposition~\ref{ausdual2} that
$\Tor_n^{\RR}(\Ext_{\RR}^n(M,\RR),M)$
can be identified to $\Hom_{\RR}(M,M)$.

\medskip

Let us sum up these observations in the following:

\begin{Theorem} \label{syzofperfect} Let $\RR$ be a Gorenstein local ring and  let $M$ be a
perfect module with a minimal free resolution $\FF$
  Let  $E$ be the module of 
$k$-syzygies of $M$, $1<k<n$.
   There exists an exact sequence
\[ 0 \rar E^*\otimes E \lar \Hom_{\RR}(E,E) \lar
\Hom_{\RR}(M,M) \rar 0.\]
\end{Theorem}

This gives the bound
\[ \nu(\Hom_{\RR}(E,E)) \leq \beta_{k}(M)\beta_{k-1}(M)+\nu(\Hom_{\RR}(M,M)).\]
If $M$ is cyclic, or $\dim M = 2 $, the estimation is easy. The
formula also shows up the case of MCM modules to be a corner case for
the general HomAB problem.

\begin{Theorem} \label{syzofperfect2} Let $\RR$ be a Gorenstein local ring and  let $M$ be a
perfect module with a minimal resolution. For the module $E$ of
$k$-syzygies of $M$, $1<k< \mbox{\rm proj dim $M$}$,   there exists an exact sequence
\[ 0 \rar E^*\otimes E \lar \Hom_{\RR}(E,E) \lar
\Hom_{\RR}(M,M) \rar 0.\]
\end{Theorem}

\begin{Corollary}
If $M=\RR/I$, $E$ is a local module.
\end{Corollary}

\begin{Corollary} Suppose $\RR$ is a Gorenstein local ring of dimension
$d$. If $M=\RR/I$ is a Cohen-Macaulay module of projective dimension
$n$ and $E$ is a module of $k$-syzygies of $M$ for $1<k<n$, then
\[ \depth \End(E)\geq d-n+1. \]
\end{Corollary}

\begin{proof} It will be enough to show that $\depth E\otimes E^*\geq
d-n+1$.

Dualizing the presentation
\[ 0\rar E \lar F_{k-1}\lar \cdots \lar F_1\lar F_0 \lar M \rar 0,\]
we obtain a projective resolution of $E^*$
\[ 0  \rar F_{0}^*\lar F_1^*\lar \cdots \lar F_{k-1}^*\lar E^*
\rar 0.\]
Note that $\depth E=d-n+k$ and $\depth E^*=d-k+1$.
Tensoring the second complex by $E$ we get a complex of modules of
depth $d-n+k$. It follows easily that the complex is exact, that thus
$\depth E\otimes E^*\geq (d-n+k)-(k-1)=d-n+1$.

\end{proof}

\subsection*{Some generic modules}\index{generic module}
Let $\varphi$ be a generic $n\times m$ matrix over a field
$k$,
$\varphi=[\xx]=[x_{ij}]$, $1\leq i\leq m, 1\leq j \leq n$, $m\geq n$,
set $\RR=k[\xx]_{(\xx)} $ and consider the
module $E$
\[ 0\rar \RR^n \stackrel{\varphi}{\lar} \RR^m \lar E \rar 0.
\]

We study the ring $\CC=\Hom_\RR(E,E)$. Let $I=I_n(\varphi)$ be the
ideal of maximal minors of $\varphi$. $\RR/I$ is a normal
Cohen--Macaulay domain and many of its properties
are deduced from the Eagon--Northcott complexes (\cite{EN}). We are
going to use them to obtain corresponding properties of $E$:

\begin{itemize}
\item If $n<m\geq 3$ (which we assume throughout), $E$ is a reflexive module.
\item By dualizing we obtain the
Auslander dual $D(E)$
 of $E$:
\[ 0\rar E^* \lar \RR^m \stackrel{\varphi^*}{\lar} \RR^n \lar
D(E)=\Ext_\RR^1(E, \RR)\rar 0.
\]

\item $\Ext_\RR^1(E,\RR)$ is isomorphic to an ideal of $\RR/I$:
$\Ext_{\RR}^1(E,\RR)$ is a perfect $\RR$--module, annihilated $I$.
These assertions follow directly from the perfection of the complexes.
Finally, localizing at $I$ the complex is quasi-isomorphic to Koszul
complex of a regular sequence which yield
\[ \Ext_\RR^1(E,\RR)_I \simeq (\RR/I)_I.\]

\end{itemize}

\begin{Theorem} $\CC=\Hom_{\RR}(E,E)$ is a local ring generated by
$\nu(E^*)\nu(E)+1=m\times {m\choose n+1}+1$ elements.
\end{Theorem}

\begin{proof}
As in the discussion of example above,
The functor $\Ext_{\RR}^1(E,X)= \Ext_{\RR}^1(E, \RR)\otimes X= \RR/I\otimes
X$, so that
  \[\Hom_{\RR}(\Ext_{\RR}^1(E,\cdot),\Ext_{\RR}^1(E,\cdot))
=\Hom_{\RR/I}(\Ext_\RR^1(E,\RR),\Ext_\RR^1(E,\RR))
=\RR/I,\]
as $\Ext_\RR^1(E,\RR)$ is isomorphic to an ideal of the normal domain
$\RR/I$.
Furthermore, since $E$ is reflexive and $E^*$ is a module of syzygies
in the Eagon--Northcott complex it has no free summand according to a
previous observation, $E$
has the same property.
The final assertion comes from the complex again.

\end{proof}

\section{Homological properties of local modules}


Let $(\RR,\m)$ be a Noetherian local ring and $E$
a finitely generated $\RR$-module. Set $\CC=\End(E)$ and  consider
the natural  action of $\CC$  making $E$ a left  $\CC$--module. In this section $\mbox{\rm mod}(\CC)$ denotes the category 
of left finitely generated $\CC$-modules.

\medskip
Let us begin with an observation.

\begin{Proposition} \label{projgen}
Let $E$ be a local module and $\CC$ its ring of endomorphisms. 
If $M$ is a nonzero  projective $\CC$-module, then $M$ is
a generator of $\mbox{\rm mod}(\CC)$.
\end{Proposition}

\begin{proof}
 Let $\JJ$ be the Jacobson
radical of $\CC$. Up to isomorphism there is only one indecomposable
$\CC/\JJ$-module, which we denote by $L$.

Let $F$ be a finitely generated left $\CC$--module. Then $M/\JJ M$
and $F/\JJ F$ are isomorphic to direct sums $L^n$ and $L^m$ of $L$.
It follows that for some integer $r>0$ there is a surjection
\[ \varphi: M^r \rar F/\JJ F. \]
Since $M$ is projective, $\varphi$ can lifted to a mapping
$\Phi: M^r\rar F$ such that $F = \Phi(M^r) + \JJ F$. By Nakayama
Lemma $\Phi$ is surjective, as desired.

\end{proof}

\begin{Theorem} \label{homollocal} Let $(\RR, \m)$ a Noetherian local ring 
and $E$ a local
module. If $\RR$ is normal domain
 and  $E$ is $\CC$--projective and
it is $\RR$-torsionfree whose rank is invertible in $\RR$ 
   then $E$ is $\RR$-free.
\end{Theorem}

\begin{proof}

According to Corollary~\ref{Jacsemi},
 $E$ has no free summand.

By Proposition~\ref{projgen}, $E$ is a projective generator and in
particular we have an isomorphism of left $\CC$-modules
\begin{eqnarray}\label{stablegen}
 E^s \simeq \CC\oplus H.
\end{eqnarray}

\medskip

To prove the assertion, we recall how the {\em trace} of the elements of $\CC$
may be defined. Let $\SS$ be the field of fractions of $\RR$.
Consider that canonical ring homomorphism
\[ \varphi: \End_\RR(E)\lar \End_\SS(\SS\otimes E).
 \] For $\ff\in \End_\RR(E)$, define
\[ \mathrm{tr}(\ff)= \mathrm{tr}_\SS(\varphi(\ff)),\]
where $\mathrm{tr}_\SS$ is the usual trace of a matrix
representation. It is independent of the chosen $\SS$-basis of
$\SS\otimes E$.

To show that $\mathrm{tr}(\ff)\in \RR$ uses a standard argument.
 For each prime ideal
$\p$ of $\RR$ of height $1$, $E_\p$ is a free $\RR_\p$-module of rank
$r=\mbox{\rm rank }(E)$,  and picking one of its basis gives also a $\SS$-basis for
$\SS\otimes E$ while showing that $\mathrm{tr}(\ff)\in \RR_\p$. Since
$\RR$ is normal, $\bigcap_\p \RR_\p=\RR$, and thus $\mathrm{tr}(\ff)\in
\RR$.

Note that this defines an element of $\Hom_\RR(\CC, \RR)$.
Now looking at the isomorphism (\ref{stablegen}), yields that the
trace of $\CC$ as an $\RR$-module is contained in the trace
ideal of $E$, which is a proper ideal of $\RR$ since $E$ has no $\RR$-free summand.
But this is impossible since $\mathrm{tr}(\II)=r$, which is a unit of
$\RR$ by assumption.

\end{proof}

\begin{Corollary}\label{mainlocalcor}
Let $\RR$ be a Noetherian local ring and  $E$  a local  $\RR$-module such that $\CC$ has finite global dimension.
\begin{enumerate}
\item \mbox{\rm (\cite[Theorem 2.5]{BH84}, \cite[Theorem 3.1]{RegAlg})}
  $\CC$ is a maximal Cohen-Macaulay
$\RR$-module. 
\item  If $E$  is a maximal Cohen-Macaulay $\RR$--module then it is a projective $\CC$--module.
\item Moreover, if
$\RR$ is a normal domain  and
  the rank $\mathrm{r}$ of $E$ is invertible in
 $\RR$  then $\RR$ is a regular local ring.
\end{enumerate}
\end{Corollary}

\begin{proof}
Since $\CC$ has a unique two-sided maximal ideal,
    by  \cite[Theorem
2.5]{BH84} or \cite[Theorem 3.1]{RegAlg}, $\CC$
 is a Cohen-Macaulay $\RR$-module. (For similar reasons, the left and the right global dimensions of $\CC$, and
 similarly  their
 finitistic versions, coincide.)
On the other hand, 
 if $\xx=x_1, \ldots, x_n$ is a maximal
$\RR$-sequence for $E$ and $\CC$, $n= \mbox{\rm gl dim }\CC$,
 $E/(\xx)E$ is a $\CC$-module with nonzero socle. It follows from
Auslander-Buchsbaum formula (\cite[Theorem 1.3.3]{BH}
(still valid for these algebras by \cite{BH84}, \cite{RegAlg})
that is
\[ \mbox{\rm proj dim}_\CC E/(\xx)E=
 \mbox{\rm proj dim}_\CC E + n =\mbox{\rm gl dim }\CC=n.\]
Thus by Theorem~\ref{homollocal} $E$ is a free $\RR$-module. This
means that
$\CC$ is a matrix ring over $\RR$ of finite global dimension
 and therefore $\RR$ is a regular
local ring.
\end{proof}

\begin{Conjecture} Let $E$ be a local $\RR$--module. If $E$ is $\CC$--cyclic of finite projective dimension over $\CC$, then $E$ is $\CC$--projective.
\end{Conjecture}

{\bf Acknowedgments: } The author was partially supported by the NSF. He is 
also grateful to the referee 
for the numerous suggestions that improved  the paper.

\end{document}